\documentclass[3p,11pt,a4paper,twoside,fleqn,sort&compress]{article}
\usepackage{amsfonts}
\usepackage{amsmath}
\usepackage{hyperref}
\usepackage{amssymb,amsthm}
\usepackage[mathscr]{euscript}
\usepackage{amssymb,amsmath,latexsym}
\usepackage[bottom]{footmisc}
\usepackage[affil-it]{authblk}
\usepackage{mathtools}

\newtheorem{theorem}{Theorem}[section]

\newtheorem{definition}[theorem]{Definition}

\newtheorem{lemma}[theorem]{Lemma}

\newtheorem{remark}[theorem]{Remark}

\numberwithin{equation}{section}
\begin{document}

\title{On a Hilfer  fractional differential equation with nonlocal Erd\'{e}lyi-Kober fractional integral boundary conditions}
\author{Mohamed I. Abbas\thanks{miabbas77@gmail.com}\\ Department of Mathematics and Computer Science, Faculty of Science, Alexandria University, Alexandria 21511, Egypt}
\date{}

\maketitle
\makeatletter
\renewcommand\@makefnmark%
{\mbox{\textsuperscript{\normalfont\@thefnmark)}}}
\makeatother

\begin{abstract}
We consider a Hilfer fractional differential equation with nonlocal Erd\'{e}lyi-Kober fractional integral boundary conditions. The existence, uniqueness and Ulam-Hyers stability results are investigated by means of the Krasnoselskii's fixed point theorem and  Banach's fixed point theorem. An example is given to illustrate the main results.
\end{abstract}

\section{Introduction}
It has become widely observed in recent years a large number of research papers interested in the theory of fractional differential equations, whether those involving classical Riemann-Liouville and Caputo type fractional derivatives or that include Hadamard and Hilfer type fractional derivatives, see for example \cite{Abbas4,Abbas3,cc6,Bashir1,Bashir2,X,Y,Rafal,T,Z,W,JinRong,Vivek} and references cited therein.\\
On the other hand, The stability of functional equations was originally raised by Ulam \cite{Ulam}, next by Hyers \cite{Hyers}. Thereafter, this type of stability is called the Ulam-Hyers stability. The concept of stability for a functional equation arises when we replace the functional equation by an inequality which acts as a perturbation of the equation. Considerable efforts have been made to study the Ulam-Hyers stability of all kinds of fractional differential equations, see for example \cite{Abbas1,Abbas2,Sabbas} and references therein.\\

In the past few years, the Erd\'{e}lyi-Kober fractional derivative, as a generalization
of the Riemann-Liouville fractional derivative, is often used, too \cite{Ntouyas16,wang14}. 
An Erd\'{e}lyi-Kober operator is a fractional
integration operation introduced by Arthur Erd\'{e}lyi and Hermann Kober in 1940 \cite{Kober}. These operators have been used by many authors, in particular, to obtain solutions of the single, dual and triple integral equations possessing special functions of mathematical physics as their kernels. In \cite{Ahmad17}, B. Ahmad et al. studied the existence and uniqueness of solution of a class of boundary value problems of Caputo fractional differential equations with Riemann-Liouville and Erd\'{e}lyi-Kober fractional integral boundary conditions of the form
$$
\begin{cases}
^{C}\mathcal{D}^{q}x(t)=f(t,x(t)),~~t\in [0,T],\\
x(0)=\alpha\mathcal{I}^{p}x(\zeta),~~x(T)=\beta\mathcal{I}_{\eta}^{\gamma,\delta}x(\xi),~~0<\zeta,\xi<T.
\end{cases}
$$
In \cite{Ntouyas17}, B. Ahmad and S. K. Ntouyas considered the following Riemann-Liouville fractional differential inclusion with Erd\'{e}lyi-Kober fractional integral boundary conditions
$$
\begin{cases}
\mathcal{D}^{q}x(t)\in F(t,x(t)),~~0<t<T,~1<q\leq 2,\\
x(0)=0,~~\alpha x(T)=\sum_{i=1}^{m}\beta_i\mathcal{I}_{\eta_i}^{\gamma_i,\delta_i}x(\xi),~~0<\xi<T,
\end{cases}
$$
they applied endpoint theory, Krasnoselskii's multi-valued fixed point theorem and Wegrzyk's fixed point
theorem for generalized contractions.\\

By using Mawhin continuation theorem, Q. Sun et al. \cite{Sun18} investigated the existence of
solutions of the following boundary value problem at resonance
$$
\begin{cases}
^{C}\mathcal{D}^{q}x(t)=f(t,x(t),x^{'}(t)),~~t\in [0,T],\\
x(0)=\alpha~\mathcal{I}_{\eta}^{\gamma,\delta}x(\zeta),~~x(T)=\beta~^{\rho}\mathcal{I}^{p}x(\xi),~~0<\zeta,\xi<T,
\end{cases}
$$
where $^{\rho}\mathcal{I}^{p}$ denotes to the generalized Riemann-Liouville (Katugampola) type integral of order $p>0$.\\

In the last of this brief survey, N. Thongsalee et al. \cite{Ntouyas15} studied the sufficient conditions for existence and uniqueness of solutions for system of Riemann-Liouville fractional differential equations subject to the nonlocal Erd\'{e}lyi-Kober fractional integral conditions of the form
$$
\begin{cases}
\mathcal{D}^{q_1}x(t)=f(t,x(t),y(t)),~~t\in [0,T],~1<q_1\leq2\\
\mathcal{D}^{q_2}y(t)=g(t,x(t),y(t)),~~t\in [0,T],~1<q_2\leq2\\
x(0)=0,~~y(T)=\sigma_1\mathcal{I}_{\eta_1}^{\gamma_1,\delta_1}x(\xi_1),~~0<\xi_1<T,\\
y(0)=0,~~x(T)=\sigma_2\mathcal{I}_{\eta_2}^{\gamma_2,\delta_2}y(\xi_2),~~0<\xi_1<T.
\end{cases}
$$
Based on the above mentioned papers, we consider the Hilfer fractional differential equations with Erd\'{e}lyi-Kober fractional integral boundary conditions of the form    
\begin{equation}\label{main}
\begin{cases}
^{H}\mathcal{D}^{\alpha,\beta}x(t)=f(t,x(t)),~~t\in [0,T],\\

x(0)=0,~~x(T)=\displaystyle \sum_{i=1}^{m}\sigma_{i}\mathcal{I}_{\eta_{i}}^{\mu_{i},\delta_{i}}x(\xi_{i}),~~
\end{cases}
\end{equation}
where $^{H}\mathcal{D}^{\alpha,\beta}$ is the Hilfer fractional derivative of order $\alpha\in(0,1)$ and type $\beta\in[0,1]$ introduced by Hilfer (see, \cite{Hilfer1,Hilfer2,Hilfer3}) , $\mathcal{I}_{\eta_{i}}^{\mu_{i},\delta_{i}}$ is the Erd\'{e}lyi-Kober fractional integral of order $\delta_{i}>0$ with $\eta_{i}>0$ and $\mu_{i}\in\mathbb{R},~i=1,2,\cdots,m$ and $\sigma_{i}\in\mathbb{R}$, $\xi_{i}\in(0,T)$ are given constants.\\
To the best of the author's knowledge this is the first paper dealing with Hilfer differential equation subject to Erd\'{e}lyi-Kober type integral boundary conditions.\\
The paper is organized as follows: Section 2 contains some preliminary concepts related to fractional calculus and Section 3 comprises the existence and uniqueness results. In Section 4, we analyze the Ulam-Hyres stability results. Finally, Section 5 contains an illustrative example of our main results.

\section{Preliminaries}
In this section we present some definitions and lemmas which will be used in our results later.\\
At first, we review some fundamental definitions of the Riemann-Liouville fractional integral and derivative which will be made up to the Hilfer fractional derivative (see \cite{Tatar,Kilbas}).
\begin{definition}
The Riemann-Liouville fractional integral of order $\alpha>0$ of a continuous function $y:(0,\infty)\to\mathbb{R}$ is defined by
\begin{equation}
\mathcal{I}^{\alpha}y(t)=\frac{1}{\Gamma(\alpha)}\int_{0}^{t}(t-s)^{\alpha-1}y(s)\ ds,~~n-1<\alpha<n,
\end{equation} 
where $n=[\alpha]+1,~[\alpha]$ denotes the integer part of a real number $\alpha$ and $\Gamma(\cdot)$ is the Gamma function defined by 
$\Gamma(\alpha)=\int_{0}^{\infty}e^{-s}s^{\alpha-1}ds$, provided the integral exists.
\end{definition}

\begin{definition}
The Riemann-Liouville fractional derivative of order $\alpha>0$ of a continuous function $y:(0,\infty)\to\mathbb{R}$ is defined by
\begin{eqnarray*}
^{RL}\mathcal{D}^{\alpha}y(t)&=&\mathcal{D}^{n}\mathcal{I}^{n-\alpha}y(t)\\
&=&\frac{1}{\Gamma(n-\alpha)}\left(\frac{d}{dt}\right)^{n}\int_{0}^{t}(t-s)^{n-\alpha-1}y(s)\ ds,~~n-1<\alpha<n,
\end{eqnarray*}
\end{definition}

\begin{definition}(Hilfer fractional derivative)\label{Hderiv}
The Hilfer fractional derivative operator of order $\alpha$ and type $\beta$ is defined by
\begin{equation}\label{Hilfer}
^{H}\mathcal{D}^{\alpha,\beta}y(t)=\mathcal{I}^{\beta(n-\alpha)}\mathcal{D}^{n}\mathcal{I}^{(1-\beta)(n-\alpha)}y(t),
\end{equation}
where $n-1<\alpha<n$, $0\leq\beta\leq 1$ and $\mathcal{D}=\frac{d}{dt}$.
\end{definition}

This generalization (\ref{Hilfer}) yields the classical Riemann-Liouville fractional derivative operator when $\beta=0$. Moreover, for $\beta=1$, it gives the Caputo fractional derivative operator.\\
Some properties and applications of the generalized Riemann-Liouville fractional derivative are given in \cite{Hilfer1}.\\

\begin{definition}
The Erd\'{e}lyi-Kober fractional integral of order $\delta>0$ with $\eta>0$ and $\mu\in\mathbb{R}$ of a continuous function $y:(0,\infty)\to\mathbb{R}$ is defined by
\begin{equation}
\mathcal{I}_{\eta}^{\mu,\delta}y(t)=\frac{\eta t^{-\eta(\delta+\mu)}}{\Gamma(\delta)}\int_{0}^{t}\frac{s^{\eta\mu+\eta-1}y(s)}{(t^{\eta}-s^{\eta})^{1-\delta}}ds,
\end{equation} 
 provided the right side is pointwise defined on $\mathbb{R}^+$.
\end{definition}

\begin{remark}
For $\eta=1$, the above operator is reduced to the Kober operator
$$K^{\mu,\delta}y(t)=\frac{t^{-(\delta+\mu)}}{\Gamma(\delta)}\int_{0}^{t}\frac{s^{\mu}y(s)}{(t-s)^{1-\delta}}ds,~~\mu,\delta>0,$$
that was introduced  for the first time by Kober in \cite{Kober}. For $\mu=0$, the Kober operator is reduced to the Riemann-Liouville
fractional integral with a power weight:
$$K^{0,\delta}y(t)=\frac{t^{-\delta}}{\Gamma(\delta)}\int_{0}^{t}\frac{y(s)}{(t-s)^{1-\delta}}ds,~~\delta>0.$$
\end{remark}

\begin{lemma}\label{Erd}
Let $\delta,\eta>0$ and $\mu,q\in\mathbb{R}$. Then we have
$$\mathcal{I}_{\eta}^{\mu,\delta}t^q=\frac{t^{q}\Gamma(\mu+(q/\eta)+1)}{\Gamma(\mu+(q/\eta)+\delta+1)}.$$
\end{lemma}

\begin{lemma}\label{integral}
Let $\eta,\lambda$ and $\nu$ be positive constants. Then
$$\int_{0}^{t}(t^{\eta}-s^{\eta})^{\lambda-1}s^{\nu-1}=\frac{t^{\eta(\lambda-1)+\nu}}{\eta}\mathbf{B}(\frac{\nu}{\eta},\lambda),$$
where $\mathbf{B}(w,v)=\int_{0}^{1}(1-s)^{w-1}s^{v-1}d,~~(\mathcal{R}e(w)>0,\mathcal{R}e(v)>0)$ is the well-known beta function.
\end{lemma}

\begin{lemma}\label{kilbas}
let $1<\alpha\leq 2$. Then
$$\mathcal{I}^{\alpha}(^{RL}\mathcal{D}^{\alpha}f)(t)=f(t)-\frac{(\mathcal{I}^{1-\alpha}f)(a)}{\Gamma(\alpha)}(t-a)^{\alpha-1}-\frac{(\mathcal{I}^{2-\alpha}f)(a)}{\Gamma(\alpha-1)}(t-a)^{\alpha-2}.$$
\end{lemma}
Now, We adopt the following definitions of Ulam-Hyeres and generalized Ulam-Hyers stabilities from Rus \cite{Rus}.

\begin{definition}\label{def1}
Equation considered in problem (\ref{main}) is Ulam-Hyers stable if there exists a real number $C_{f}>0$ such that for each $\epsilon>0$ and for each solution $y\in C([0,T],\mathbb{R})$ of the inequality 
$$
\left|^{H}\mathcal{D}^{\alpha,\beta}y(t)-f\left(t,y(t)\right)\right|\leq\epsilon,~~t\in [0,T],
$$
 there exists a solution $x\in C([0,T],\mathbb{R})$ of Eq.(\ref{main}) with $$\left|y(t)-x(t)\right|\leq C_{f}\epsilon,~~~t\in [0,T].$$ 
\end{definition}

\begin{definition}\label{def2}
Equation considered in problem (\ref{main}) is generalized Ulam-Hyers stable if there exists $\vartheta_{f}\in C(\mathbb{R}^{+},\mathbb{R}^{+})$, $\vartheta_{f}(0)=0$ such that for each solution $y\in C([0,T],\mathbb{R})$ of the inequality 
$$
\left|^{H}\mathcal{D}^{\alpha,\beta}y(t)-f\left(t,y(t)\right)\right|\leq\epsilon,~~t\in [0,T],
$$
there exists a solution $x\in C([0,T],\mathbb{R})$ of Eq.(\ref{main}) with $$\left|y(t)-x(t)\right|\leq \vartheta_{f}(\epsilon),~~~t\in [0,T].$$ 
\end{definition}

\begin{remark}
It is clear that Definition \textnormal{\ref{def1}} $\Longrightarrow$ Definition \textnormal{\ref{def2}}.
\end{remark}
To end this section, we recall the Krasnoselskii's fixed point theorem, which plays a key role in the main results for the problem (\ref{main}).
\begin{theorem}(\textbf{Krasnoselskii's fixed point theorem~\cite{Smart}})\label{Krasno}
Let $K$ be a closed convex and non-empty subset of a Banach space $\mathbb{X}$. Let $\mathcal{A}$ and $\mathcal{B}$, be two operators such that 
\begin{itemize}
  \item [(i)] $\mathcal{A}x+\mathcal{B}y\in K$, for all $x,y\in K$;
  \item [(ii)] $\mathcal{A}$ is a contraction mapping;
  \item [(iii)] $\mathcal{B}$ is compact and continuous.
\end{itemize}
Then there exists a $z\in K$ such that $z=\mathcal{A}z+\mathcal{B}z$.
\end{theorem}
\section{Existence and Uniqueness Results}
Let $C([0,T],\mathbb{R})$ be the Banach space of all real-valued continuous functions from $[0,T]$ into $\mathbb{R}$ equipped by the norm $\|x\|_{C}=\sup_{t\in [0,T]}|x(t)|,~\forall x\in C([0,T],\mathbb{R})$.

\begin{lemma}\label{ess}
let $1<\alpha< 2,0\leq\beta\leq1,\gamma=\alpha+2\beta-\alpha\beta,\delta_{i},\eta_{i}>0,\mu_{i},\sigma_{i}\in\mathbb{R},\xi_{i}\in(0,T),~i=1,2,\cdots,m$ and $h\in C([0,T],\mathbb{R})$. Then the linear Hilfer fractional differential equation subject to the Erd\'{e}lyi-Kober fractional integral boundary conditions    
\begin{equation}\label{linear}
\begin{cases}
^{H}\mathcal{D}^{\alpha,\beta}x(t)=h(t),~~t\in [0,T],\\
x(0)=0,~~x(T)=\sum_{i=1}^{m}\sigma_{i}\mathcal{I}_{\eta_{i}}^{\mu_{i},\delta_{i}}x(\xi_{i}),~~
\end{cases}
\end{equation}
is equivalent to the following fractional integral equation
\begin{equation}\label{equivalent}
x(t)=\mathcal{I}^{\alpha}h(t)+\frac{t^{\gamma-1}}{\Delta}\left(\sum_{i=1}^{m}\sigma_{i}\mathcal{I}_{\eta_{i}}^{\mu_{i},\delta_{i}}\mathcal{I}^{\alpha}h(\xi_{i})-\mathcal{I}^{\alpha}h(T)\right),
\end{equation}
where 
\begin{equation}
\Delta=T^{\gamma-1}-\sum_{i=1}^{m}\sigma_{i}\xi_{i}^{\gamma-1}\frac{\Gamma(\mu_{i}+(\gamma-1)/\eta_{i}+1)}{\Gamma(\mu_{i}+(\gamma-1)/\eta_{i}+\delta_{i}+1)}\not=0.
\end{equation}
\end{lemma}

\begin{proof}
By Definition \ref{Hderiv} (with $n=2$), the equation $^{H}\mathcal{D}^{\alpha,\beta}x(t)=h(t)$ can be written as
\begin{equation}\label{nequal2}
\mathcal{I}^{\beta(2-\alpha)}\mathcal{D}^{2}\mathcal{I}^{(1-\beta)(2-\alpha)}x(t)=h(t).
\end{equation}
Applying the Riemann-Liouville fractional integral $\mathcal{I}^{\alpha}$ of order $\alpha$ to the both sides of the equation (\ref{nequal2}), we get
$$\mathcal{I}^{\alpha}\mathcal{I}^{\beta(2-\alpha)}\mathcal{D}^{2}\mathcal{I}^{(1-\beta)(2-\alpha)}x(t)=\mathcal{I}^{\alpha}h(t).$$
Indeed,
$$\mathcal{I}^{\alpha}\mathcal{I}^{\beta(2-\alpha)}\mathcal{D}^{2}\mathcal{I}^{(1-\beta)(2-\alpha)}x(t)=\mathcal{I}^{\gamma}\mathcal{D}^{2}\mathcal{I}^{2-\gamma}x(t)=\mathcal{I}^{\gamma}(^{RL}\mathcal{D}^{\gamma}x)(t),$$
thus
$$\mathcal{I}^{\gamma}(^{RL}\mathcal{D}^{\gamma}x)(t)=\mathcal{I}^{\alpha}h(t).$$
By using Lemma \ref{kilbas} (with $a=0$), we get
$$x(t)=\mathcal{I}^{\alpha}h(t)+\frac{(\mathcal{I}^{1-\gamma}x)(0)}{\Gamma(\gamma)}t^{\gamma-1}+\frac{(\mathcal{I}^{2-\gamma}x)(0)}{\Gamma(\gamma-1)}t^{\gamma-2}.$$
Setting $(\mathcal{I}^{1-\gamma}x)(0)=c_1,~(\mathcal{I}^{2-\gamma}x)(0)=c_2$ gives
$$x(t)=\mathcal{I}^{\alpha}h(t)+\frac{c_1}{\Gamma(\gamma)}t^{\gamma-1}+\frac{c_2}{\Gamma(\gamma-1)}t^{\gamma-2}.$$
From the first boundary condition $x(a) = 0$, we obtain $c_2 = 0$. Then we get
\begin{equation}\label{subst}
x(t)=\mathcal{I}^{\alpha}h(t)+\frac{c_1}{\Gamma(\gamma)}t^{\gamma-1}.
\end{equation}
In view of Lemma \ref{Erd} and the boundary condition $\displaystyle x(T)=\sum_{i=1}^{m}\sigma_{i}\mathcal{I}_{\eta_{i}}^{\mu_{i},\delta_{i}}x(\xi_{i})$, we get
\begin{eqnarray*}
\mathcal{I}^{\alpha}h(T)+\frac{c_1}{\Gamma(\gamma)}T^{\gamma-1}&=&\sum_{i=1}^{m}\sigma_{i}\mathcal{I}_{\eta_{i}}^{\mu_{i},\delta_{i}}\left(\mathcal{I}^{\alpha}h(\xi_{i})+\frac{c_1}{\Gamma(\gamma)}\xi_{i}^{\gamma-1}\right)\\
&=&\sum_{i=1}^{m}\sigma_{i}\mathcal{I}_{\eta_{i}}^{\mu_{i},\delta_{i}}\mathcal{I}^{\alpha}h(\xi_{i})+\frac{c_1}{\Gamma(\gamma)}\sum_{i=1}^{m}\sigma_{i}\mathcal{I}_{\eta_{i}}^{\mu_{i},\delta_{i}}\xi_{i}^{\gamma-1}\\
&=&\sum_{i=1}^{m}\sigma_{i}\mathcal{I}_{\eta_{i}}^{\mu_{i},\delta_{i}}\mathcal{I}^{\alpha}h(\xi_{i})+\frac{c_1}{\Gamma(\gamma)}\sum_{i=1}^{m}\sigma_{i}\xi_{i}^{\gamma-1}\frac{\Gamma(\mu_{i}+(\gamma-1)/\eta_{i}+1)}{\Gamma(\mu_{i}+(\gamma-1)/\eta_{i}+\delta_{i}+1)}.
\end{eqnarray*}
Therefore, we conclude that 
$$c_1=\Gamma(\gamma)\left(\frac{\sum_{i=1}^{m}\sigma_{i}\mathcal{I}_{\eta_{i}}^{\mu_{i},\delta_{i}}\mathcal{I}^{\alpha}h(\xi_{i})-\mathcal{I}^{\alpha}h(T)}{T^{\gamma-1}-\sum_{i=1}^{m}\sigma_{i}\xi_{i}^{\gamma-1}\frac{\Gamma(\mu_{i}+(\gamma-1)/\eta_{i}+1)}{\Gamma(\mu_{i}+(\gamma-1)/\eta_{i}+\delta_{i}+1)}}\right)=\frac{\Gamma(\gamma)}{\Delta}\left(\sum_{i=1}^{m}\sigma_{i}\mathcal{I}_{\eta_{i}}^{\mu_{i},\delta_{i}}\mathcal{I}^{\alpha}h(\xi_{i})-\mathcal{I}^{\alpha}h(T)\right).$$
By substitution the value of $c_1$ in equation (\ref{subst}), we obtain the solution (\ref{equivalent}). The converse follows by direct computation. This completes the proof.
\end{proof}

We consider the following assumptions:
\begin{itemize} 
  \item []$(H1)$ The function $f: [0,T]\times\mathbb{R}\to \mathbb{R}$ is continuous. 
  \item []$(H2)$ There exist  constants $L,M>0$  such that
  $$|f(t,x)-f(t,y)|\leq L|x-y|,~\text{for each}~t\in [0,T],~x,y\in C([0,T],\mathbb{R}),$$ and $$M=\sup_{t\in [0,T]}|f(0,t)|.$$
  \item []$(H3)$ There exists a function $\psi\in C([0,T],\mathbb{R}^+)$ such that
  $$|f(t,x)|\leq \psi(t),~\text{for all}~(t,x)\in [0,T]\times\mathbb{R},$$ and $$\|\psi\|=\sup_{t\in [0,T]}|\psi(t)|.$$
\end{itemize}

We transform the problem (\ref{main}) into a fixed point problem $\mathcal{F}x=x$, where the operator $\mathcal{F}:C([0,T],\mathbb{R})\to C([0,T],\mathbb{R})$ is defined by
$$(\mathcal{F}x)(t)=\mathcal{I}^{\alpha}f(s,x(s))(t)+\frac{t^{\gamma-1}}{\Delta}\left(\sum_{i=1}^{m}\sigma_{i}\mathcal{I}_{\eta_{i}}^{\mu_{i},\delta_{i}}\mathcal{I}^{\alpha}f(s,x(s))(\xi_{i})-\mathcal{I}^{\alpha}f(s,x(s))(T)\right),$$
where
$$\mathcal{I}_{\eta_{i}}^{\mu_{i},\delta_{i}}\mathcal{I}^{\alpha}f(s,x(s))(\xi_{i})=\frac{\eta_{i}\xi_{i}^{-\eta_{i}(\delta_{i}+\mu_{i})}}{\Gamma(\delta_{i})\Gamma(\alpha)}\int_{0}^{\xi_{i}}\int_{0}^{y}\frac{y^{\eta_{i}\mu_{i}+\eta_{i}-1}(y-s)^{\alpha-1}}{(\xi_{i}^{\eta_{i}}-y^{\eta_{i}})^{1-\delta_{i}}}f(s,x(s))dsdy,$$\\
where $\xi_{i}\in(0,T)$ for $i=1,2,\cdots,m$, and
$$\mathcal{I}^{\alpha}f(s,x(s))(y)=\frac{1}{\Gamma{\alpha}}\int_{0}^{y}(y-s)^{\alpha-1}f(s,x(s))ds,~~y\in\{t,T\}$$
for $t\in[0,T].$\\
The following uniqueness result is based on Banach's fixed point theorem.
\begin{theorem}\label{uniqueness}
Under the assumptions $(H1)$ and $(H2)$, the boundary value problem (\ref{main}) has a unique solution on $[0,T]$, provided that 
$L\Omega<1,$ where
\begin{equation}\label{condcontra}
\Omega=\frac{1}{\Gamma(\alpha+1)}\left(T^{\alpha}+\frac{T^{\gamma+\alpha-1}}{|\Delta|}+\frac{T^{\gamma-1}}{|\Delta|}\sum_{i=1}^{m}\frac{|\sigma_{i}|\xi_{i}^{\alpha}\Gamma(\alpha/\eta_{i}+\mu_{i}+1)}{\Gamma(\delta_{i}+\alpha/\eta_{i}+\mu_{i}+1)}\right)
\end{equation}
\end{theorem}

\begin{proof}
Define the set $\mathcal{B}_r=\{x\in C([0,T],\mathbb{R}):\|x\|_{C}\leq r\}$ with
$$r\geq\frac{M\Omega}{1-L\Omega}.$$
Clearly, the fixed points of the operator $\mathcal{F}$ are solutions of problem (\ref{main}).\\
We show that $\mathcal{F}\mathcal{B}_{r}\subset\mathcal{B}_{r}$. For any $x\in\mathcal{B}_{r}$, we have
\begin{eqnarray*}
|(\mathcal{F}x)(t)|&\leq&\sup_{t\in [0,T]}\left\{\mathcal{I}^{\alpha}|f(s,x(s))|(t)+\frac{t^{\gamma-1}}{|\Delta|}\mathcal{I}^{\alpha}|f(s,x(s))|(T)\right.\\
&+&\left.\frac{t^{\gamma-1}}{|\Delta|}\sum_{i=1}^{m}|\sigma_{i}|\mathcal{I}_{\eta_{i}}^{\mu_{i},\delta_{i}}\mathcal{I}^{\alpha}|f(s,x(s))|(\xi_{i})\right\}\\
&\leq&\mathcal{I}^{\alpha}\left(|f(s,x(s))-f(s,0)|+|f(s,0)|\right)(T)\\
&+&\frac{T^{\gamma-1}}{|\Delta|}\mathcal{I}^{\alpha}\left(|f(s,x(s))-f(s,0)|+|f(s,0)|\right)(T)\\
&+&\frac{T^{\gamma-1}}{|\Delta|}\sum_{i=1}^{m}|\sigma_{i}|\mathcal{I}_{\eta_{i}}^{\mu_{i},\delta_{i}}\mathcal{I}^{\alpha}\left(|f(s,x(s))-f(s,0)|+|f(s,0)|\right)(\xi_{i})\\
&\leq&(Lr+M)\left(\frac{1}{\Gamma(\alpha)}\int_{0}^{T}(T-s)^{\alpha-1}ds+\frac{T^{\gamma-1}}{|\Delta|\Gamma(\alpha)}\int_{0}^{T}(T-s)^{\alpha-1}ds\right.\\
&+&\left.\frac{T^{\gamma-1}}{|\Delta|\Gamma(\alpha)}\sum_{i=1}^{m}|\sigma_{i}|\frac{\eta_{i}\xi_{i}^{-\eta_{i}(\delta_{i}+\mu_{i})}}{\Gamma(\delta_{i})}\int_{0}^{\xi_{i}}\int_{0}^{y}\frac{y^{\eta_{i}\mu_{i}+\eta_{i}-1}(y-s)^{\alpha-1}}{(\xi_{i}^{\eta_{i}}-y^{\eta_{i}})^{1-\delta_{i}}}dsdy\right)\\
&=&\frac{Lr+M}{\Gamma(\alpha+1)}\left(T^{\alpha}+\frac{T^{\gamma+\alpha-1}}{|\Delta|}+\frac{T^{\gamma-1}}{|\Delta|}\sum_{i=1}^{m}|\sigma_{i}|\frac{\eta_{i}\xi_{i}^{-\eta_{i}(\delta_{i}+\mu_{i})}}{\Gamma(\delta_{i})}\int_{0}^{\xi_{i}}\frac{y^{\alpha+\eta_{i}\mu_{i}+\eta_{i}-1}}{(\xi_{i}^{\eta_{i}}-y^{\eta_{i}})^{1-\delta_{i}}}dy\right)\\
&=&\frac{Lr+M}{\Gamma(\alpha+1)}\left(T^{\alpha}+\frac{T^{\gamma+\alpha-1}}{|\Delta|}+\frac{T^{\gamma-1}}{|\Delta|}\sum_{i=1}^{m}\frac{|\sigma_{i}|\xi_{i}^{\alpha}\Gamma(\alpha/\eta_{i}+\mu_{i}+1)}{\Gamma(\delta_{i}+\alpha/\eta_{i}+\mu_{i}+1)}\right)\\
&=&(Lr+M)\Omega\leq r,
\end{eqnarray*}
which implies that $\mathcal{F}\mathcal{B}_{r}\subset\mathcal{B}_{r}$.\\

Next, for each $t\in [0,T]$ and $x,y\in C([0,T],\mathbb{R})$, , we have
\begin{eqnarray*}
|(\mathcal{F}x)(t)-(\mathcal{F}y)(t)|&\leq&\mathcal{I}^{\alpha}\left(|f(s,x(s))-f(s,y(s))|\right)(T)+\frac{T^{\gamma-1}}{|\Delta|}\mathcal{I}^{\alpha}\left(|f(s,x(s))-f(s,y(s))|\right)(T)\\
&+&\frac{T^{\gamma-1}}{|\Delta|}\sum_{i=1}^{m}|\sigma_{i}|\mathcal{I}_{\eta_{i}}^{\mu_{i},\delta_{i}}\mathcal{I}^{\alpha}\left(|f(s,x(s))-f(s,y(s))|\right)(\xi_{i})\\
&\leq&\frac{L}{\Gamma(\alpha+1)}\left(T^{\alpha}+\frac{T^{\gamma+\alpha-1}}{|\Delta|}+\frac{T^{\gamma-1}}{|\Delta|}\sum_{i=1}^{m}\frac{|\sigma_{i}|\xi_{i}^{\alpha}\Gamma(\alpha/\eta_{i}+\mu_{i}+1)}{\Gamma(\delta_{i}+\alpha/\eta_{i}+\mu_{i}+1)}\right)\|x-y\|\\
&=&L\Omega\|x-y\|,
\end{eqnarray*}
which implies that $\|\mathcal{F}x-\mathcal{F}y|\leq L\Omega\|x-y\|$. As $L\Omega<1$, $\mathcal{F}$ is contraction.\\
Therefore, we deduce by the Banach's contraction mapping principle, that $\mathcal{F}$ has a fixed point which is the unique solution of the boundary value problem (\ref{main}). The proof is completed.
\end{proof}\qedhere

The following existence theorem is based on the Krasnoskelskii's fixed point theorem (Theorem \ref{Krasno}).

\begin{theorem}\label{existence}
Assume that assumptions $(H1)-(H3)$ hold. Then the boundary value problem (\ref{main}) has at least one solution on $[0,T]$, provided that $L\Lambda<1,$ where
\begin{equation}
\Lambda=\frac{1}{\Gamma(\alpha+1)}\left(\frac{T^{\gamma+\alpha-1}}{|\Delta|}+\frac{T^{\gamma-1}}{|\Delta|}\sum_{i=1}^{m}\frac{|\sigma_{i}|\xi_{i}^{\alpha}\Gamma(\alpha/\eta_{i}+\mu_{i}+1)}{\Gamma(\delta_{i}+\alpha/\eta_{i}+\mu_{i}+1)}\right)
\end{equation}
\end{theorem}
\begin{proof}
Consider $\mathcal{B}_{r^{*}}=\{x\in C([0,T],\mathbb{R}):\|x\|_{C}\leq r^{*}\}$ with $r^{*}\geq\|\psi\|\Omega$. We define two operators $\mathcal{A},\mathcal{B}$ on $\mathcal{B}_{r^{*}}$ by
$$(\mathcal{A}x)(t)=\frac{t^{\gamma-1}}{\Delta}\left(\sum_{i=1}^{m}\sigma_{i}\mathcal{I}_{\eta_{i}}^{\mu_{i},\delta_{i}}\mathcal{I}^{\alpha}f(s,x(s))(\xi_{i})-\mathcal{I}^{\alpha}f(s,x(s))(T)\right),$$
and
$$(\mathcal{B}x)(t)=\mathcal{I}^{\alpha}f(s,x(s))(t).$$
For each $t\in [0,T]$ and any $x,y\in\mathcal{B}_{r^{*}}$, we have
\begin{eqnarray*}
|(\mathcal{A}x)(t)+(\mathcal{B}x)(t)|&\leq&\sup_{t\in [0,T]}\left\{\mathcal{I}^{\alpha}|f(s,x(s))|(t)+\frac{t^{\gamma-1}}{|\Delta|}\mathcal{I}^{\alpha}|f(s,x(s))|(T)\right.\\
&+&\left.\frac{t^{\gamma-1}}{|\Delta|}\sum_{i=1}^{m}|\sigma_{i}|\mathcal{I}_{\eta_{i}}^{\mu_{i},\delta_{i}}\mathcal{I}^{\alpha}|f(s,x(s))|(\xi_{i})\right\}\\
&\leq&\frac{\|\psi\|}{\Gamma(\alpha+1)}\left(T^{\alpha}+\frac{T^{\gamma+\alpha-1}}{|\Delta|}+\frac{T^{\gamma-1}}{|\Delta|}\sum_{i=1}^{m}\frac{|\sigma_{i}|\xi_{i}^{\alpha}\Gamma(\alpha/\eta_{i}+\mu_{i}+1)}{\Gamma(\delta_{i}+\alpha/\eta_{i}+\mu_{i}+1)}\right)\\
&=&\|\psi\|\Omega\leq r^{*}.
\end{eqnarray*}
Therefore, $\mathcal{A}x+\mathcal{B}x\in\mathcal{B}_{r^{*}}$.\\

Next, it is easy to show that $\mathcal{A}x$ is contraction. Indeed, 
\begin{eqnarray*}
|(\mathcal{A}x)(t)-(\mathcal{A}y)(t)|&\leq&\frac{T^{\gamma-1}}{|\Delta|}\mathcal{I}^{\alpha}\left(|f(s,x(s))-f(s,y(s))|\right)(T)\\
&+&\frac{T^{\gamma-1}}{|\Delta|}\sum_{i=1}^{m}|\sigma_{i}|\mathcal{I}_{\eta_{i}}^{\mu_{i},\delta_{i}}\mathcal{I}^{\alpha}\left(|f(s,x(s))-f(s,y(s))|\right)(\xi_{i})\\
&\leq&\frac{1}{\Gamma(\alpha+1)}\left(\frac{T^{\gamma+\alpha-1}}{|\Delta|}+\frac{T^{\gamma-1}}{|\Delta|}\sum_{i=1}^{m}\frac{|\sigma_{i}|\xi_{i}^{\alpha}\Gamma(\alpha/\eta_{i}+\mu_{i}+1)}{\Gamma(\delta_{i}+\alpha/\eta_{i}+\mu_{i}+1)}\right)\|x-y\|\\
&=&L\Lambda\|x-y\|.
\end{eqnarray*}
Since $L\Lambda<1$, then $\mathcal{A}$ is contraction.\\
It remains to prove the continuity and compactness of $\mathcal{B}$. In view of assumption $(H1)$, the continuity of the function $f$ implies that the operator  $\mathcal{B}$ is continuous. Also, we observe that
\begin{eqnarray*}
|(\mathcal{B}x)(t)|&\leq&\sup_{t\in [0,T]}\left\{\mathcal{I}^{\alpha}|f(s,x(s))|(t)\right\}\\
&\leq&\frac{T^{\alpha}}{\Gamma(\alpha+1)}\|\psi\|.
\end{eqnarray*} 
This shows that $\mathcal{B}$ is uniformly bounded on $\mathcal{B}_{r^{*}}$.\\

Now, we prove the compactness of $\mathcal{B}$. We define $$\sup_{(t,x)\in [0,T]\times\mathcal{B}_{r^{*}}}|f(t,x)|=\hat{f}<\infty.$$
For each $t_1,t_2\in [0,T],t_1\leq t_2$ and $x\in\mathcal{B}_{r^{*}}$, we get
\begin{eqnarray*}
|(\mathcal{B}x)(t_2)-(\mathcal{B}x)(t_1)|&\leq&\frac{1}{\Gamma(\alpha)}\left|\int_{0}^{t_2}(t_2-s)^{\alpha-1}f(s,x(s))ds-\int_{0}^{t_1}(t_1-s)^{\alpha-1}f(s,x(s))ds\right|\\
&\leq&\frac{1}{\Gamma(\alpha+1)}\left(\int_{0}^{t_2}[(t_2-s)^{\alpha-1}-(t_1-s)^{\alpha-1}]|f(s,x(s))|ds\right.\\
&+&\left.\int_{t_1}^{t_2}(t_1-s)^{\alpha-1}|f(s,x(s))|ds\right)\\
&\leq&\frac{\hat{f}}{\Gamma(\alpha+1)}|t_{2}^{\alpha}-t_{1}^{\alpha}|.
\end{eqnarray*} 
The right hand side of the above inequality tends to zero as $t_2-t_1\to 0$, which implies that $\mathcal{B}$ is equicontinuous. Hence $\mathcal{B}$ is relatively compact on $\mathcal{B}_{r^{*}}$. By the Arzel\'{a}-Ascoli theorem, we deduce that the operator $\mathcal{B}$ is compact. We conclude, by the Krasnoskelskii's fixed point theorem, that the boundary value problem (\ref{main}) has at least one solution on $[0,T]$. The proof is completed.
\end{proof}\qedhere

\section{Stability Results}
In this section, we discuss the Ulam-Hyers and generalized Ulam-Hyers stability results for the problem (\ref{main}). 

\begin{remark}\label{remarkg}
A function $y\in C([0,T],\mathbb{R})$ is a solution of the inequality
$$
\left|^{H}\mathcal{D}^{\alpha,\beta}y(t)-f\left(t,y(t)\right)\right|\leq\epsilon,~~t\in [0,T],
$$
if and only if there exist a function $g\in C([0,T],\mathbb{R})$ (which depend on y) such that
\begin{itemize}
  \item [(i)] $|g(t)|\leq\epsilon,~t\in [0,T]$,
  \item [(ii)] $^{H}\mathcal{D}^{\alpha,\beta}y(t)=f\left(t,y(t)\right)+g(t),~~t\in [0,T]$,
\end{itemize}
\end{remark}

\begin{lemma}
If $y\in C([0,T],\mathbb{R})$ is a solution of the inequality 
$$
\left|^{H}\mathcal{D}^{\alpha,\beta}y(t)-f\left(t,y(t)\right)\right|\leq\epsilon,~~t\in [0,T],
$$
then $y$ satisfies
\begin{equation}
|y(t)-(\mathcal{F}y)(t)|\leq\Omega\epsilon,
\end{equation}
where $\Omega$ is defined in (\ref{condcontra}).
\end{lemma}
\begin{proof}
From Remark \ref{remarkg} and Lemma \ref{ess}, we have
$$y(t)=\mathcal{I}^{\alpha}(f(s,y(s))+g)(t)+\frac{t^{\gamma-1}}{\Delta}\left(\sum_{i=1}^{m}\sigma_{i}\mathcal{I}_{\eta_{i}}^{\mu_{i},\delta_{i}}\mathcal{I}^{\alpha}(f(s,y(s))+g)(\xi_{i})-\mathcal{I}^{\alpha}(f(s,y(s))+g)(T)\right).$$
Then, we get
\begin{eqnarray*}
|y(t)-(\mathcal{F}y)(t)|&=&\left|\mathcal{I}^{\alpha}(f(s,y(s))+g)(t)+\frac{t^{\gamma-1}}{\Delta}\left(\sum_{i=1}^{m}\sigma_{i}\mathcal{I}_{\eta_{i}}^{\mu_{i},\delta_{i}}\mathcal{I}^{\alpha}(f(s,y(s))+g)(\xi_{i})\right.\right.\\
&-&\mathcal{I}^{\alpha}(f(s,y(s))+g)(T)\Bigg)\\
&-&\left.\mathcal{I}^{\alpha}f(s,y(s))(t)-\frac{t^{\gamma-1}}{\Delta}\left(\sum_{i=1}^{m}\sigma_{i}\mathcal{I}_{\eta_{i}}^{\mu_{i},\delta_{i}}\mathcal{I}^{\alpha}f(s,y(s))(\xi_{i})-\mathcal{I}^{\alpha}f(s,y(s))(T)\right)\right|\\
&\leq&\mathcal{I}^{\alpha}|g|(t)+\frac{t^{\gamma-1}}{|\Delta|}\left(\sum_{i=1}^{m}|\sigma_{i}|\mathcal{I}_{\eta_{i}}^{\mu_{i},\delta_{i}}\mathcal{I}^{\alpha}|g|(\xi_{i})-\mathcal{I}^{\alpha}|g|(T)\right)\\
&\leq&\Omega\epsilon.
\end{eqnarray*}
This completes the proof.
\end{proof}\qedhere

\begin{theorem}
Assume that assumptions $(H1)$ and $(H2)$ are satisfied. Then the problem (\ref{main}) is Ulam-Hyers stable.
\end{theorem}
\begin{proof}Let $\epsilon>0$, $y\in C([0,T],\mathbb{R})$ be a solution of the inequality 
$$
\left|^{H}\mathcal{D}^{\alpha,\beta}y(t)-f\left(t,y(t)\right)\right|\leq\epsilon,~~t\in [0,T],
$$
and let $x\in C([0,T],\mathbb{R})$ be the unique solution of problem (\ref{main}). Then, we have
\begin{eqnarray*}
|y(t)-x(t)|&=&\left|y(t)-\mathcal{I}^{\alpha}f(s,x(s))(t)-\frac{t^{\gamma-1}}{\Delta}\left(\sum_{i=1}^{m}\sigma_{i}\mathcal{I}_{\eta_{i}}^{\mu_{i},\delta_{i}}\mathcal{I}^{\alpha}f(s,x(s))(\xi_{i})-\mathcal{I}^{\alpha}f(s,x(s))(T)\right)\right|\\
&=&|y(t)-(\mathcal{F}x)(t)|\\
&=&|y(t)-(\mathcal{F}y)(t)+(\mathcal{F}y)(t)-(\mathcal{F}x)(t)|\\
&\leq&|y(t)-(\mathcal{F}y)(t)|+|(\mathcal{F}y)(t)-(\mathcal{F}x)(t)|\\
&\leq&\Omega\epsilon+L\Omega|y-x|,
\end{eqnarray*}
which implies that $$|y(t)-x(t)|\leq\frac{\Omega\epsilon}{1-L\Omega},~~L\Omega<1.$$
By setting $C_{f}=\frac{\Omega}{1-L\Omega}$, we get $$|y(t)-x(t)|\leq C_{f}\epsilon.$$
Thus, the problem (\ref{main}) is Ulam-Hyers stable.\\
 If we set $\vartheta_{f}(\epsilon)=C_{f}\epsilon,~\vartheta_{f}(0)=0$, then the problem (\ref{main}) is generalized Ulam-Hyers stable.
\end{proof}\qedhere

\section{An example}
In this section we consider the following Hilfer fractional differential equation with Erd\'{e}lyi-Kober fractional integral boundary condition:
\begin{equation}\label{example}
\begin{cases}
^{H}\mathcal{D}^{\frac{4}{3},\frac{5}{6}}x(t)=\frac{|x(t)|}{25\sqrt{4+t^{2}}(1+|x(t)|)},~~t\in [0,1],\\
\\
x(0)=0,~~x(1)=\frac{1}{3}\mathcal{I}_{\frac{1}{5}}^{\frac{1}{4},\frac{3}{7}}x(\frac{5}{4})+\frac{2}{5}\mathcal{I}_{\frac{2}{9}}^{\frac{2}{3},\frac{5}{8}}x(\frac{3}{2})+\frac{5}{6}\mathcal{I}_{\frac{1}{3}}^{\frac{1}{6},\frac{1}{5}}x(\frac{2}{7}),
\end{cases}
\end{equation}
where $\alpha=\frac{4}{3},\beta=\frac{5}{6},\gamma=\frac{17}{9},T=1,m=3,\sigma_{1}=\frac{1}{3},\sigma_{2}=\frac{2}{5},\sigma_{3}=\frac{5}{6},\mu_{1}=\frac{1}{4},\mu_{2}=\frac{2}{3},\mu_{3}=\frac{1}{6},\delta_{1}=\frac{3}{7},\delta_{2}=\frac{5}{8},\delta_{3}=\frac{1}{5},\eta_{1}=\frac{1}{5},\eta_{2}=\frac{2}{9},\eta_{3}=\frac{1}{3},\xi_{1}=\frac{5}{4},\xi_{2}=\frac{3}{2},\xi_{3}=\frac{2}{7}$ and the function $f(t,x(t))=\frac{|x(t)|}{25\sqrt{4+t^{2}}(1+|x(t)|)}$.\\

We can see that
\begin{eqnarray*}
|f(t,x(t)-f(t,y(t)|&=&\left|\frac{|x(t)|}{25\sqrt{4+t^{2}}(1+|x(t)|)}-\frac{|y(t)|}{25\sqrt{4+t^{2}}(1+|y(t)|)}\right|\\
&\leq&\frac{1}{25\sqrt{4+t^{2}}}\frac{|x(t)|-|y(t)|}{(1+|x(t)|)(1+|y(t)|)}\\
&\leq&\frac{1}{50}|x-y|,
\end{eqnarray*}
which implies, by assumption $(H2)$, that $L=\frac{1}{50}$.\\
Simple calculations give
$$\Delta=T^{\gamma-1}-\sum_{i=1}^{m}\sigma_{i}\xi_{i}^{\gamma-1}\frac{\Gamma(\mu_{i}+(\gamma-1)/\eta_{i}+1)}{\Gamma(\mu_{i}+(\gamma-1)/\eta_{i}+\delta_{i}+1)}\approx -0.029801394\not=0,$$
$$\Omega=\frac{1}{\Gamma(\alpha+1)}\left(T^{\alpha}+\frac{T^{\gamma+\alpha-1}}{|\Delta|}+\frac{T^{\gamma-1}}{|\Delta|}\sum_{i=1}^{m}\frac{|\sigma_{i}|\xi_{i}^{\alpha}\Gamma(\alpha/\eta_{i}+\mu_{i}+1)}{\Gamma(\delta_{i}+\alpha/\eta_{i}+\mu_{i}+1)}\right)\approx 43.74995072,$$
and
$$\Lambda=\frac{1}{\Gamma(\alpha+1)}\left(\frac{T^{\gamma+\alpha-1}}{|\Delta|}+\frac{T^{\gamma-1}}{|\Delta|}\sum_{i=1}^{m}\frac{|\sigma_{i}|\xi_{i}^{\alpha}\Gamma(\alpha/\eta_{i}+\mu_{i}+1)}{\Gamma(\delta_{i}+\alpha/\eta_{i}+\mu_{i}+1)}\right)\approx 42.91006582.$$
Hence, we get $L\Omega\approx 0.8749990144<1$ and $L\Lambda\approx 0.8582013164<1$.\\
Therefore, the conclusion of Theorem \ref{existence} implies that the boundary value problem (\ref{example}) has at least one solution on $[0,1]$ and by Theorem \ref{uniqueness}, this solution is unique.

\end{document}